\documentclass[12pt,a4paper]{article}
\usepackage{amsmath}
\usepackage{amsthm}
\usepackage{amssymb}
\usepackage{color}
\usepackage{graphicx}
\usepackage{subfig}
\usepackage{wrapfig}
\usepackage[margin=1in,marginpar=1in,nohead]{geometry}
\usepackage[singlespacing]{setspace}
\usepackage{endnotes}
\usepackage{natbib}
\usepackage{url}
\usepackage[utf8]{inputenc}
\usepackage[pdftex,colorlinks=true,linkcolor=blue,citecolor=blue]{hyperref}
\usepackage[small,bf]{titlesec}
\usepackage{pdfpages}
\usepackage{mathrsfs}
\usepackage[normalem]{ulem}

\newcommand{\blue}{\color{blue}}
\newcommand{\black}{\color{black}}

\newcommand{\rmd}{{\mathrm d}}
\newcommand{\Var}{{\mathrm{Var}}}
\newcommand{\Cov}{{\mathrm{Cov}}}

\newcommand{\calG}{{\mathcal G}}

\usepackage[textwidth=30mm]{todonotes}

\newcounter{figurecounter}
\title{The Design of Optimal Dependency and Rewards 
\thanks{Solan acknowledges the support of the ISF grant No. 211/22. Chang acknowledges the support of the NSFC grant No. 72403040.}}
\author
{Eilon Solan\footnote{School of Mathematical Sciences, Tel-Aviv University, Tel-Aviv, Israel, 6997800, E-mail: eilons@tauex.tau.ac.il.}\and 
Avraham Tabbach\footnote{Faculty of Law, Tel-Aviv University, Tel-Aviv, Israel, 6997800, E-mail: adtabbac@tauex.tau.ac.il.}\and 
Chang Zhao\footnote{Department of Economics, University of International Business and Economics, China, E-mail: zhaochangtd@hotmail.com.}}
\date{\today}

\begin{document}


\begin{titlepage}
\maketitle
\thispagestyle{empty}

\begin{abstract}
  \noindent

We analyze a two-period principal-agent model in which the principal faces a budget constraint, and the agent's private costs of performing tasks across the two periods may be correlated. We examine the optimal design of the reward scheme and the cost correlation structure. Our findings reveal that when the budget is low, the optimal reward scheme employs \textit{sufficient performance targeting}, rewarding the agent's first performance. Conversely, when the principal's budget is high, the focus shifts to \textit{sustained performance targeting}, compensating the agent's second performance. Introducing a negative cost correlation proves particularly beneficial in both scenarios: it increases the likelihood of the agent performing at least once under low budgets and balances the agent's total costs to facilitate consistent performance under high budgets. However, the optimal cost correlation structure can be more elaborate, especially for intermediate budget levels. Our results offer valuable insights for real-world applications, such as research funding allocation.

\end{abstract}

\end{titlepage}

\onehalfspacing

\newcommand{\RNum}[1]{\lowercase\expandafter{\romannumeral #1\relax}}
\newtheorem{proposition}{Proposition}
\newtheorem{lemma}{Lemma}
\newtheorem{theorem}{Theorem}
\newtheorem{corollary}{Corollary}
\newtheorem{assumption}{Assumption}
\newtheorem{remark}{Remark}
\newtheorem{result}{Result}
\newtheorem{definition}{Definition}
\theoremstyle{definition}
\newtheorem{example}{Example}[section]

\section{Introduction}
\label{sec: introduction}

Understanding dynamic incentives is central in economics, and there is a substantial body of literature examining how best to structure rewards in dynamic principal-agent problems. For instance, the seminal papers of \citet{lazear1981agency} and \citet{harris1982theory} argue in favor of back-loading rewards, while \citet{holmstrom1987aggregation} highlight the effectiveness of linear reward allocations. 

A common assumption in this literature is that the principal’s ability to reward the agent is unconstrained, allowing for flexible incentive allocations. 
However, more often than not,
principals operate under budget constraints that limit the total resources available for rewards. This raises the question of how to optimally allocate a fixed budget to effectively motivate agents over multiple periods. We address this question by examining a scenario in which an agent’s task costs are correlated over time, and the principal, who does not know the agent’s costs, must distribute a fixed budget across periods to maximize performance.

We find that the total amount of available funds critically shapes the optimal design of rewards, introducing an important factor that has been largely overlooked in the existing literature.

Budget constraints and the need to motivate agents to exert effort over multiple periods are not merely theoretical; they arise frequently in practice. For instance, research funding agencies typically allocate grants under fixed budgets over a given time horizon, and many agencies make continued funding contingent on a midterm review that evaluates progress toward agreed performance targets. In doing so, they must decide how to allocate limited funds across project phases, such as a midterm evaluation and a final award, in order to incentivize sustained effort throughout the project’s duration.

In this paper, we study a two‑period model in which an agent can perform a task in each period, incurring privately observed costs \((c_1,c_2)\). A principal with a fixed budget \(w\) chooses how to allocate rewards across periods to maximize expected performance. In our benchmark, \(c_1\) and \(c_2\) are independently and uniformly distributed on \([0,1]\). We show that the total budget \( w \) plays a crucial role in determining the optimal reward scheme.

When \( w \) is low, it is optimal for the principal to concentrate on the agent’s first performance, regardless of whether it occurs in period one or two, by offering a high reward. We refer to this rewarding rule as \textit{sufficient performance targeting}. Because the total reward is fixed, a high reward for the first performance implies a low reward for the second performance, reducing the likelihood that 
an
agent will perform again. 
This rewarding rule is optimal 
because
with a low budget and independently distributed costs, it is improbable the agent will have low costs in both periods, making it more effective to focus resources on securing at least one performance.

By contrast, when \(w\) is high, the principal’s optimal rewarding rule offers a small reward for the first performance and reserves most of the budget for a substantial reward for the second performance, thereby incentivizing the agent to perform twice. We call this rewarding rule \textit{sustained performance targeting}. In this setup, the agent’s incentives are close to an “all or nothing” decision, since most of the reward is paid only after two performances. The higher budget increases the probability that the agent can afford to perform in both periods, thereby making a reward scheme that fosters continued effort optimal.

In our benchmark model, we assume the agent’s costs are independent across periods. We then extend the model to allow for more flexible cost structures, introducing the possibility that costs across periods may be correlated. We consider a scenario in which the principal can select an agent with a particular cost structure from a predefined, monotonically ranked set of cost dependencies, utilizing the concept of copulas. In this extended framework, the principal can not only design the rewarding rule but also choose among different cost correlation patterns.\footnote{Throughout, we use the terms “dependency” and “correlation” interchangeably to denote any form of statistical association between two random variables.} 

This extension is particularly relevant in environments 
in which
principals encounter a wide variety of agents, each with distinct cost correlation. Some agents may exhibit positively correlated costs across periods, others negatively correlated costs, while still others may face independent costs. In the grant-allocation context, for example, it is well documented that women are more strongly affected than men by work–family conflict, particularly when young children are present (see \citealp{BraunsteinBercovitz2013,LivingstonJudge2008}). This suggests that women with young children may experience greater intertemporal fluctuations in the opportunity cost of research. Heterogeneity in cost correlation can also arise from other sources. Disciplinary focus is one such factor: empirical researchers whose work is organized around a certain data set tend to face more persistent cost structures across periods, whereas theory-oriented researchers who frequently shift among ideas or modeling approaches may exhibit more dispersed costs over time.

Our findings show that the optimal rewarding rule identified in the benchmark case remains largely robust in this extended framework. Moreover, for a wide range of parameter values, negative intertemporal correlation in the agent’s costs is especially valuable for the principal. Interpreted in a grant-allocation setting, these results suggest that funding agencies may improve the effectiveness of their programs by prioritizing disciplines or project types in which researchers face greater intertemporal variation in research costs.

We next explain how cost correlation affects the agent’s incentives. When the budget is low, the optimal rewarding rule remains \textit{sufficient performance targeting}: the principal offers a high reward for the agent’s first performance. Negative cost correlation is helpful for two reasons. First, it eliminates the option value of waiting: an agent who draws a low cost today does not expect another low draw tomorrow and thus has no incentive to delay
performance.
Second, by inversely linking high and low costs across periods, negative dependence raises the probability that at least one period’s cost falls below the prize, increasing the likelihood of at least one performance.

When the budget is high, the principal switches to \textit{sustained performance targeting} by placing a large prize on the second performance. Here, negative dependence helps through a different channel: it reduces the variance of the total cost \(c_1+c_2\), so the prize covers the agent’s two‑period cost for a larger set of types, making the reward more effective.

To illustrate this point, assume the agent’s costs \((c_1, c_2)\) are each uniformly distributed on \([0,1]\), the total budget is \(w = 1\), and the principal adopts sustained performance targeting by allocating all rewards to the second performance. Under perfect positive correlation, \(c_2=c_1\), so an agent performs twice 
if and only if
\(c_1+c_2=2c_1\le 1\), i.e., \(c_1\le \tfrac{1}{2}\); thus exactly one half of 
the
agents complete both performances. Under perfect negative correlation, \(c_2=1-c_1\), so \(c_1+c_2= 1\) and every agent performs twice. Hence the fraction of agents who perform twice rises from $\frac{1}{2}$ to \(1\).

It is worth noting that our finding—namely, that introducing a negative dependency between the agent’s costs can be advantageous when the goal is \textit{sustained performance targeting}—parallels the literature on bundling, where multiple products are sold as a single package. It is well established that bundling tends to be more profitable than selling products separately when consumer valuations are negatively correlated (see, e.g., \citealp{Schmalensee1984}; \citealp{Long1984}; \citealp{McAfee1989}; \citealp{Chu2011}; \citealp{ChenRiordan2013}). In our setting, negative cost dependency similarly balances the agent’s total cost over time, increasing the likelihood of sustained performance and improving the efficiency of the reward allocation.

When the budget is at intermediate levels, however, our analysis indicates that independent or even positive dependency between the agent’s costs may be optimal. We leave a detailed explanation of this case to the main text.

After analyzing the case in which the principal selects an agent from a predefined, monotonically ranked family of cost correlations, we 
drop this restriction and allow the principal to choose \emph{any} joint distribution of costs across periods, as long as the marginal distribution in each period remains uniform. In this more general setting, we view a
\emph{scheme}
as a pair consisting of a reward rule and a cost correlation, and we fully characterize all 
optimal
schemes.

A first implication is that every optimal scheme uses one of two extreme reward rules: a purely sufficient rule, which allocates the entire budget to a single performance, or a purely sustained rule, which allocates the entire budget only after two performances. The purely sufficient rule, supported by a suitably negative cost correlation, is optimal only when the budget is relatively small. 
The purely sustained rule is optimal for all budget levels, 
but 
requires a 
specific form of cost correlation that cannot be described as simply positive or negative. In particular, for every agent who is induced to perform, the two-period costs must add up exactly to the budget.
Because this cost structure is highly tailored, it is unlikely to arise in practice and is therefore mainly of theoretical interest.

Throughout this paper, we assume that at each history, the principal offers a single reward to all agent types and does not screen agents by offering contracts. If the principal could screen agents, the problem would align with the literature on dynamic screening. Owing to the complexity of dynamic contracting, analysis in that field has largely been restricted to environments where a “first-order approach” is applicable \citep{battaglini2019optimal}. In our model, the cost correlation across periods can be arbitrary, and the budget constraint is exogenously imposed. Consequently, the first-order approach does not apply, making the problem challenging to solve using traditional dynamic screening techniques.

The paper is organized as follows. Section \ref{sec: related literature} discusses related literature. Section \ref{section:model} presents the model. 
Section \ref{iidsection} studies the benchmark case where the costs are independent.
Section \ref{copuanaly} studies the optimal cost correlation under
the Farlie-Gumbel-Morgenstern (FGM)
copulas family. 
Section \ref{optsection} studies the optimal cost correlation under no restrictions. 
Section \ref{sec: conclusion} concludes.


\section{Related Literature.}
\label{sec: related literature}

The allocation of rewards in dynamic principal-agent problems---whether to back-load them or implement a linear structure---has been a central topic of debate in the literature. Researchers have examined various contexts in which different approaches may be optimal, depending on the characteristics of the principal-agent relationship and the nature of the tasks involved.

The seminal papers of \citet{lazear1981agency} and \citet{harris1982theory} argue in favor of back-loading rewards. \citet{lazear1981agency} suggests that delaying compensation effectively deters the agent from prematurely ending the employment relationship, as doing so would mean forfeiting substantial future rewards. This creates strong incentives for the agent to remain with the firm and maintain effort over time. \citet{harris1982theory} provide a different rationale, emphasizing the role of risk sharing. In their model, both the principal and the agent learn more about the agent's ability over time. The back-loading of rewards emerges as an optimal strategy due to its insurance effect: a smooth, increasing wage path helps spread risk over time, ensuring that the risk-averse agent is not overly exposed to income volatility, which would otherwise necessitate a higher wage premium.

In contrast, \citet{holmstrom1987aggregation} provide compelling arguments for linear allocation of rewards. They focus on situations where the agent's ability is known, and the agent exerts effort to influence the performance outcome, which is subject to random noise. They demonstrate that for risk-averse agents, linear contracts are effective because they provide a clear and straightforward link between compensation and performance without requiring detailed monitoring of how the agent allocates effort across multiple tasks. The linearity in compensation strikes a balance between providing sufficient incentives and offering insurance against risks beyond the agent's control. 

\citet{sannikov2008continuous} extends these insights to a continuous-time framework. He shows that over short horizons the optimal contract is approximately linear in incremental output. Over longer horizons, however, optimal compensation can display rich non-linear dynamics in wages and effort. In particular, when the agent’s continuation value is low, the contract may back-load compensation for a reason similar to \citet{lazear1981agency}.

Most studies on dynamic principal-agent problems with private information typically assume that the agent's type is either fixed over time or independently and identically distributed across periods. 
Recent works have started to explore the implications of dependency in the agent's type across periods.

\citet{GuoHorner2020} examine a principal-agent model where the agent's private value for a good evolves according to a two-state Markov chain. The agent always desires the good, while the principal aims to allocate it only when the value is high. Without monetary transfers, the optimal allocation rule has the following properties: when the agent's value is high, the good is allocated; when the value is low, the principal offers future entitlements to deter the agent from falsely claiming a high value. As value persistence increases, a low value today has a greater impact on future expectations, forcing the principal to promise more in the future for the low-value agent to forgo the current unit, ultimately reducing overall efficiency. Interestingly, negative correlations are beneficial: a stronger negative correlation enhances the efficiency of future promises because preference misalignment is only temporary.

There are two key differences between \citet{GuoHorner2020} and our paper. First, \citet{GuoHorner2020} do not allow for monetary transfers, whereas our model permits the principal to make transfers. Second, in their study, allocating the good is efficient only when the agent's type is high, so the challenge lies in deterring low types from claiming the good. In contrast, our paper considers a setting where performance is always desirable under complete information, making it unclear which type of agent should optimally be induced to perform. As a result, although both papers find that less persistence in the agent’s type is beneficial, the underlying reasoning differs fundamentally.

Our paper is also connected to the dynamic screening literature, where the agent's type is correlated across periods. In this context, the principal offers the agent a contract to screen their type (e.g., \citealp{battaglini2005long}; \citealp{pavan2014dynamic}; \citealp{battaglini2019optimal}). Due to the complexity of dynamic contracts, analysis has been confined to environments where a form of the ``first-order approach'' is applicable. \citet{battaglini2019optimal} provide conditions under which this approach works. Specifically, they show that the first-order approach fails when the frequency of principal-agent interactions is sufficiently high---or equivalently, when the discount factor, time horizon, and type persistence are sufficiently large. In our project, the cost correlation across periods can be arbitrary, and there is an exogenously given budget constraint. Consequently, the problem is challenging to solve using dynamic screening techniques.

Intertemporal dependence can also arise from the technology of multi-stage tasks, rather than from correlation in costs. \citet{GreenTaylor2016} study optimal contracting for a multistage project in which the agent privately observes intermediate progress and the principal must sustain incentives until completion. They show that optimal schemes can employ history-dependent instruments, such as soft deadlines and termination rules, together with the strategic use of self-reported progress to maintain effort across stages. Although their setting features moral hazard with hidden progress (rather than privately known task costs), it shares with our paper the central theme of designing dynamic incentives when the principal values sustained success over time.

Our paper is also related to work in the auction and procurement framework that studies dynamic incentives when tasks are purchased sequentially and the relevant cost environment evolves over time. \citet{JofreBonetPesendorfer2014} analyze sequential procurement auctions and show that the cost-minimizing auction format depends on whether the two items are complements or substitutes. \citet{CisternasFigueroa2015} study sequential procurement when the first-period winner can invest before the second auction to improve future costs, illustrating how dynamic mechanisms trade off current procurement costs against future incentive provision. Finally, \citet{CheGale1998} show that private budget constraints can substantially alter performance comparisons across standard auction formats. While these papers focus on competition among multiple bidders, whereas we study a single-agent reward scheme with an exogenously given transfer budget, they share the broader message that intertemporal dependence and budget limitations are crucial for optimal dynamic design.

Our paper is also related to the literature on bundling. It is well recognized that mixed bundling is profitable when consumer valuations are negatively correlated (e.g., \citealp{Schmalensee1984}; \citealp{Long1984}; \citealp{McAfee1989}; \citealp{Chu2011}).
However, these studies focus on specific cases. \citet{ChenRiordan2013} introduce a copula approach to model the joint distribution of consumer valuations for goods and provide sufficient conditions for the profitability of bundling. They demonstrate that mixed bundling is generally more profitable than separate selling when valuations for the two products are negatively dependent. The intuition is 
as follows: introducing a small bundle discount from optimal separate-selling prices reduces profits from consumers already buying both goods but increases profits from those switching from purchasing only one good to buying both. Under negative dependence---where high valuations for one product align with low valuations for the other---the gains outweigh the losses. Our result aligns with this literature: when the principal's goal is to induce double performance and has a high budget, a negative cost correlation is optimal, as it combines a high cost with a low cost, thereby averaging out the total cost of performing twice.

Another related paper is \citet{AuChen2021}, who study a team production problem and find that maximizing skill diversity among workers in teams is optimal. The benefit of worker heterogeneity arises from peer effects: by pairing high-ability agents capable of imposing significant sanctions with low-ability agents prone to moral hazard---since ability and effort are complementary---the principal effectively leverages peer monitoring. This negative assortative matching minimizes overall incentive costs, thereby benefiting the principal.

Finally, our paper is closely related to \citet{leshem2022option}, 
who examine an enforcement agency regulating the behavior of potential violators in a two-period model with constrained sanctions. They make two restrictive assumptions that, when translated into our framework, correspond to the following:
(a) the agent's costs are independent over time, and
(b) the compensation to an agent who performs once is the same,
regardless of whether the performance occurs in the first or second period.
\citet{leshem2022option} show that the optimal sanction scheme maximizes sanctions for repeat offenders when the total sanction is high, and exhibits leniency for repeat offenders when it is low.

\section{Model}
\label{section:model}

Two risk-neutral players, a principal (she) and an agent (he) interact for two periods. Players do not discount their payoffs. In each period the agent chooses an action $e\in \{P, N\}$, where $P$ stands for \emph{Perform} and $N$ stands for \emph{Not perform}.
The action chosen by the agent is observed by the principal.
The cost for the agent for $N$ is zero. The cost for the agent for $P$ is random and may vary among the periods. Moreover, the cost for $P$ is the agent's private information.
Denote by $A$ and $B$ the cost for the agent in periods 1 and 2, respectively.
The way $A$ and $B$ are determined will be explained shortly.

To induce the agent to perform, the principal can offer the agent a rewarding rule,
which is a function $r$ from the set of all finite histories (which consists of the agent's past actions) to the non-negative reals: at each history $h$, the agent obtains $r(h)$ from the principal.
We assume that the principal has a total budget $w$.
This implies that the sum of the offered reward after any two consecutive finite histories cannot exceed $w$.
We also assume that unused budget is valueless for the principal.

The last ingredient of the model is the way the costs $A$ and $B$ are selected. 
The principal has at her disposal a set $\calG$ of joint distributions of pairs of costs $(A,B)$,
and she is free to choose any distribution $G \in \calG$.
This assumption, that the principal can choose the cost dependence structure, is natural in environments with a heterogeneous pool of agents who differ in characteristics such as gender, disciplinary focus, or career stage. Each profile of characteristics induces a distinct pattern of cost dependence, so by selecting an agent 
with a particular profile, the principal effectively chooses a cost dependence structure and can tailor the corresponding rewarding rule.

We assume that for every $G\in\calG$, the marginal distribution of the cost in each period is \emph{uniform} on $[0,1]$. This normalization keeps fixed the distribution of costs in each period and allows us to isolate the role of intertemporal dependence in shaping incentives. Without it, comparisons across dependence structures would be less informative, because the principal would prefer distributions under which costs are more likely to be low.

We study three cases. In the first (Section~\ref{iidsection}), $\calG$ contains a single distribution in which the two period costs are independent. In the second (Section~\ref{copuanaly}), $\calG$ consists of a parametric family of joint distributions whose dependence varies monotonically with a single parameter. In the third (Section~\ref{optsection}), $\calG$ contains all joint distributions on $[0,1]^2$ with uniform marginals.

The goal of the principal is to maximize the performance level (i.e., the expected number of times the agent performs),
and the goal of the agent is to maximize the difference between his total reward and total cost.
For simplicity, we assume that whenever an agent is indifferent between performing and not performing, he performs.\footnote{Because the agent's type is continuous, the tie-breaking rule is used only with probability 0.}  

Figure~\ref{fig: timeline} describes the timeline of the game:
First the principal chooses the cost distribution $G$ and announces the rewarding rule $r$ to the agent. The agent then privately learns his period 1 cost $A$ and decides whether to perform in period 1. 
Then, the agent
privately learns his period 2 cost $B$ and decides whether to perform in period 2. Finally, the agent obtains the rewards, depending on the rewarding rule and the agent's past actions.

We refer to the two choices $(G, r)$ made by the principal as a \emph{scheme}. 
Our goal is to study the 
optimal scheme
-- 
the scheme that maximizes the expected agent's performance given the budget constraint,
and to understand the driving forces behind it.

\begin{figure}[ht]
\includegraphics[scale=0.5]{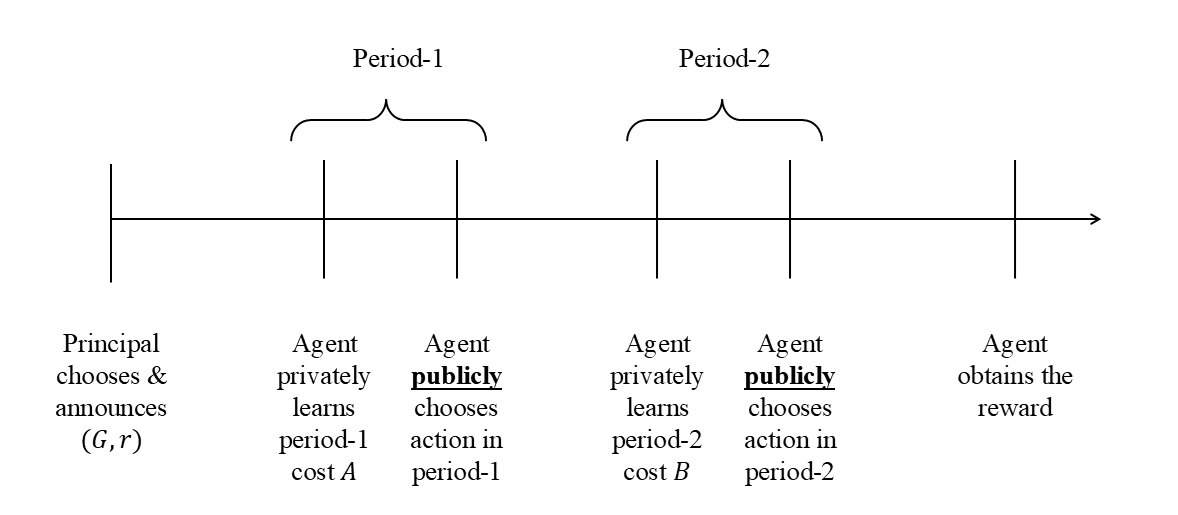}
\centering
\caption{The timeline of the game.}
\label{fig: timeline}
\end{figure}

When the principal’s budget satisfies \(w \geq \tfrac{3}{2}\), she can choose \(G\) such that \(A\) and \(B\) are i.i.d. and commit to paying the entire budget $w$ only if the agent performs in both periods. 
In this way, she can induce the agent to perform in both periods. 
Indeed, conditional on having performed in period~1, performing in period~2 yields net payoff \(w-B \geq w-1>0\), while not performing yields \(0\); hence the agent performs in period~2 whenever he performed in period~1. Anticipating this, in period~1 the agent’s expected payoff from performing is \(w-A-\mathbb{E}[B]=w-A-\tfrac{1}{2}\), which is nonnegative for every \(A\in[0,1]\) when \(w\geq \tfrac{3}{2}\). The agent therefore performs in period~1, and consequently also in period~2.
Hence, we focus on the non-trivial case with the following assumption.

\begin{assumption}
$w< \frac{3}{2}$.
\end{assumption}

\begin{remark}
\begin{enumerate}
\item
The assumption that the function $r$ is non-negative captures the idea that the agent has limited liability, so that the agent cannot be penalized for not performing.\color{black}
\item 
We assume that unused budget is valueless for the principal. This fits scenarios where the performance of the agent is not comparable with monetary gain,\footnote{For example, in an environmental protection context, it is difficult to weigh the benefits of reducing pollutants against monetary gains.} 
or the principal has no control over the residual reward.\footnote{For instance, unused budget is returned to the government at the end of the year.}
\item In our model the principal selects one out of a pool of agents and maximizes the expected number of performances of that agent. An alternative interpretation is that the principal selects a sub-population with the same cost distribution, and tries to maximize overall expected performances of that sub-population.
\end{enumerate}
\end{remark}
\color{black}

\subsection{Simple properties of the optimal rewarding rule}

It can be shown that to maximize the performance level, it is w.l.o.g. to focus on rewarding rules under which action \color{black} $N$ (at every history) is not rewarded. With this restriction, every rewarding rule can be characterized by three parameters $(x, y, z)$, where $x = r(P)$  is the reward the agent obtains for choosing $P$ in period 1; $y = r(N,P)$ is the reward the agent obtains for choosing $P$ in period 2 if he chose $N$ in period 1; and $z = r(P,P)-r(P)$ is the reward the agent obtains for choosing $P$ for a second time, see Figure \ref{fig: uniform zero}. Note that $z$ captures the incremental reward from a second performance, and the agent's total reward by choosing $P$ in both periods is $x+z=r(P, P)$.

\begin{figure}[ht]
\includegraphics[scale=0.8]{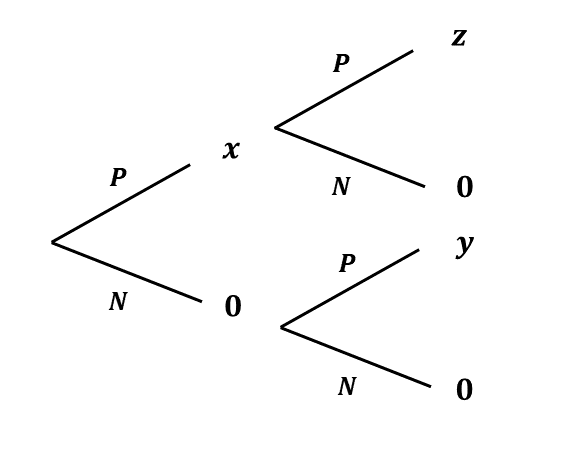}
\centering
\caption{The rewarding rule.}
\label{fig: uniform zero}
\end{figure}

We start with a simple observation that bounds the maximum number of performances. In each period, an agent with type $c>w$ does not benefit from performing in that period (regardless of what he does in the other period). Therefore, the maximum possible performance in each period is $F(w)=w$,
where $F$ is the CDF of the uniform distribution on $[0,1]$.
This provides an upper-bound on the performance level.

\begin{lemma}
\label{upperboundle}
In equilibrium, the agent with type $c>w$ in a certain period does not perform in that period. 
Hence, the upper-bound on performance level is $2w$.
\end{lemma}

The next lemma states that we can assume w.l.o.g.~that the reward to an agent who performed twice is the full budget $w$.

\begin{lemma}
\label{nowastetwo}
It is without loss of generality to assume that under the optimal rewarding rule we have $x+z=w$.
\end{lemma}

\begin{proof}
See Appendix \ref{pnowastetwo}.
\end{proof}

The intuition for this result is as follows. Suppose $x+z<w$, and consider the effect of increasing $z$ to $w-x$. 
If the agent performs in period 1, then \color{black} the probability that he will perform in period 2 does not decrease, since his reward for the second performance increases. 
If the agent does \color{black} not perform in period 1
when $z$ is lower, \color{black} then an increase in $z$ makes performing in period 1 more favorable than before, and the agent may decide to perform in period 1 (and later in period 2 as well). As a result, it is always optimal for the principal to fully use her budget to reward \color{black} two performances.

Note that the above argument does not apply to $y$, and it is not clear what is the effect of 
increasing
$y$ (fixing $x$ and $z$) on the overall performance level. Indeed, as $y$ increases, two effects take place: On the one hand, if the agent does not perform in period 1, a higher level of $y$ increases the probability the agent will perform in period 2. On the other hand, a higher level of $y$ discourages an agent from performing in period 1, since by waiting for another period the agent may get a higher reward.

\section{The Benchmark Model --- I.I.D.}
\label{iidsection}

We begin by analyzing the benchmark model, in which the agent's costs in the two periods are independent.  
Formally, 
$\calG$ contains a single distribution:
the one in which $A$ and $B$ are independent.
We will characterize the optimal rewarding rule $(x, y, z)$ in this case.

In period 2 an agent performs if and only if 
$B$
is lower than the period 2 reward: That is, 
$B\leq z$
if he performed in period 1, and 
$B\leq y$
if he did not perform in period 1. Hence, if the agent with cost $c_1$ performs in period 1, his expected payoff is 
\begin{equation}
\label{equ:payoff:1}
(x-c_1)+\int_0^{\min(z, 1)}(z-c)\, \rmd c,
\end{equation}
and if he does not perform in period 1, his expected payoff is 
\begin{equation}
\label{equ:payoff:2}
0+\int_0^{\min(y, 1)}(y-c)\, \rmd c.
\end{equation}

An agent with cost $c_1$ will perform in period 1 if and only if his payoff from performing
in period 1 (given by Eq.~\eqref{equ:payoff:1}) 
is higher than his payoff from not performing in period 1 (given by Eq.~\eqref{equ:payoff:2}):

\begin{equation}
 \overset{\text{Agent performs in period 1}}{\overbrace{%
(x-c_1)+\int_0^{\min(z, 1)} (z-c)\, \rmd c}} \quad \geq \text{ }\overset{%
\text{Agent does not perform in period 1}}{\overbrace{%
0+\int_0^{\min(y, 1)} (y-c)\, \rmd c}}.  
\label{eq: payoff for plaintiff}
\end{equation}

The cost threshold $\overline{c}$ for period 1 performance can be solved by finding the minimal $c_1$ for which Eq.~\eqref{eq: payoff for plaintiff} holds.

The principal chooses $(x, y, z)$ to maximize
\begin{equation}
 \overset{\text{period 1}}{\overbrace{%
 F(\overline{c})}} + \text{ }\overset{%
\text{period 2, } P \text{ in period 1}}{\overbrace{%
F(\overline{c}) \cdot F(z)}} + \text{} \overset{\text{period 2, $N$ in period 1}}{\overbrace{%
 \big(1-F(\overline{c})\big) \cdot F(y)}},
\label{uniformula}
\end{equation}
subject to the constraints

\[
  x+z = w, \quad  0 \leq x \leq w, \quad 0 \leq y \leq w, \quad and \quad 0 \leq z \leq w.
\]

The following proposition characterizes the solution of the maximization problem (\ref{uniformula}).

\begin{proposition}
\label{uniformop}
     Suppose the cost of performing the task in each period is independently and uniformly distributed on $[0,1]$. 
     Then an optimal rewarding rule is given by
\begin{equation}
(x^*,y^*,z^*) =
 \begin{cases}
 (w, w, 0), &\text{ for } 0 \leq w \leq 0.6,\\
 (w-g(w), w, g(w)), &\text { for } 0.6 \leq w \leq 1,\\
 (0, h(w), w), & \text { for } 1 \leq w \leq 1.4, \\
 (0, 0, w), & \text { for } 1.4 \leq w \leq 1.5,
 \end{cases}
 \label{opiid}
\end{equation}
where 
\begin{gather}
  g(w)=\frac{1}{3}(w+1-2 \sqrt{w^{2}-2.5 w+1.75}), \\
  h(w)=\frac{1}{3}(w+1-2 \sqrt{w^{2}+0.5 w-1.25}).
\end{gather}
\end{proposition}
When $w \leq 1$, the rewarding rule in \eqref{opiid} is uniquely optimal. 
When $w > 1$, a rewarding rule is optimal if and only if it is of the form $(w-z,\,y^*,\,z)$ for some $z \geq 1$; in particular, no rewarding rule outside this class is optimal.

\begin{proof}
See Appendix \ref{puniformop}.
\end{proof}

The function $g$
increases from 0 to $\frac{1}{3}$ for $w\in[0.6, 1]$, and the function \black $h$ decreases from $\frac{1}{3}$ to 0 for $w\in[1, 1.4]$. See Figure \ref{optimalRewards} for a graphical illustration of the optimal rewarding rule. In particular, when $w=1$, both rewarding rules \black $(x=\frac{2}{3}, y=1, z=\frac{1}{3})$ and $(x=0, y=\frac{1}{3}, z=1)$ are optimal. 

\color{black}

\begin{figure}[ht]
\includegraphics[scale=0.6]{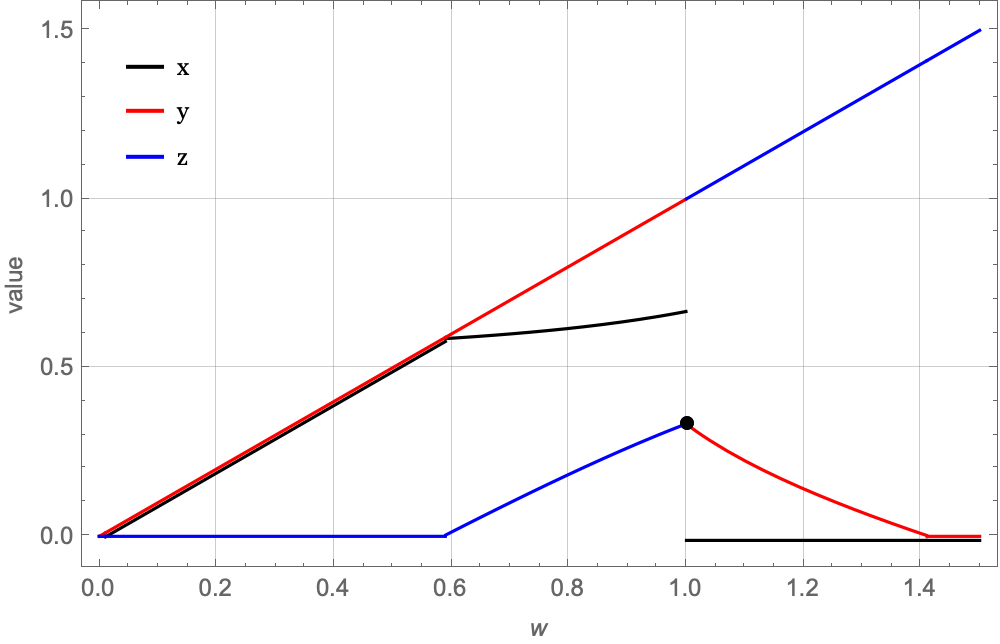}
\centering
\caption{Optimal rewarding rule in Proposition~\ref{uniformop}.}
\label{optimalRewards}
\end{figure}

What we learn from Proposition \ref{uniformop} is that when the budget \( w \) is low, it is optimal to set \( x \) and \( y \) large, and \( z \) small. In this situation, the principal’s focus is on inducing a single performance, and the agent’s first performance—whether in period one or period two—receives a substantial reward. We refer to this strategy as \emph{sufficient performance targeting}, where the principal’s aim is to secure at least one instance of performance.

We can further divide this case into two subcases. If \( w \in (0, 0.6) \), the optimal rule takes the form \((x = w,\, y = w,\, z = 0)\), which we refer to as \emph{purely sufficient}, indicating that \emph{only} the first performance is rewarded. For \( w \in (0.6, 1) \), the principal still emphasizes sufficient performance targeting but 
has enough budget to encourage a
second performance; the optimal rule becomes \emph{partially sufficient}, allocating a positive amount also to the second performance.

When the budget \( w \) is large, it is optimal to set \( x \) and \( y \) small, and \( z \) large. The principal’s focus 
shifts to inducing two performances, and the agent receives a significant reward only if he performs in both periods. We refer to this strategy as \emph{sustained performance targeting}, where the goal is to foster ongoing effort by offering a more substantial reward for the second performance.

We can also divide this scenario into two subcases. When \( w \in (1.4, 1.5) \), the optimal rule takes the form \((x = 0,\, y = 0,\, z = w)\), which is referred to as \emph{purely sustained}, indicating that only the second performance is rewarded. If \( w \in (1, 1.4) \), many agent types do not perform in the first period, making it costly to set \( y = 0 \). The optimal rule is thus \emph{partially sustained}, assigning a positive reward even if the agent only performs in period 2. 

In summary, when \( w \) is low, the optimal rewarding rule is \emph{purely sufficient} targeting. Indeed, with a limited budget and independently distributed costs, it is unlikely that the agent will have low costs in both periods, making it more effective to concentrate resources on ensuring at least one performance. Conversely, when \( w \) is high, the optimal rewarding rule is \emph{purely sustained} targeting. This is because the larger budget raises the probability that the agent can afford to perform in both periods, thereby favoring a scheme that encourages continued effort. For budgets in the middle, the principal must balance the two extremes, resulting in a mixture of sufficient and sustained strategies—\emph{partially sufficient} or \emph{partially sustained}—that allocates
rewards across both performances in a more nuanced way.

\black

\section{Copulas and the Role of Dependence Level on Overall Performance.}\label{copuanaly}

The previous section studies the case where the agent's costs in the two periods are independent. In this section, we assume that the principal can select the joint cost distribution from a predefined set of distributions, with varying degrees of dependency among the two periods: A parameter $\theta$ is used to capture the monotonic change of dependence level. The objective is to determine the optimal \emph{scheme}, defined as a combination of $\theta$ and rewarding rule \((x,y,z)\).

\subsection{Modeling dependence level --- FGM copulas}

We employ copulas to capture the dependence level of costs across two periods. A \emph{copula} links the marginal distributions of random variables to create a joint distribution, facilitating the manipulation of dependence while maintaining constant marginal distributions. When the marginal distributions are uniform, the copula is synonymous with the multivariate cumulative distribution of $(A,B)$.

Our analysis centers on the Farlie-Gumbel-Morgenstern (FGM) copula family with uniform marginal distributions. The joint distribution under FGM copula is defined as
\begin{equation}
    G_{\theta\color{black}}(a, b)=P_{\theta}(A\leq a, B\leq b) = ab\bigl(1 + \theta(1 - a)(1 - b)\bigr), \ \ \ a,b \in [0,1],\color{black}
    \label{eq: FGM Copula}
\end{equation}
where \(\theta\) ranges from $-1$ to 1. The parameter \(\theta\) indicates the level of dependence between \(A\) and \(B\): a more positive (resp., negative) \(\theta\) corresponds to \color{black} a stronger positive (resp., negative) correlation.

The FGM copula is widely used in finance, economics, and insurance, largely because of its analytical simplicity and tractability; see, for example, the applications surveyed in 
\citet{hutchinson1990continuous,Conway2014FGM}. 
Due to its inherent properties, this family can model only relatively weak dependence: for the FGM copula in (\ref{eq: FGM Copula}), the Spearman correlation between the two random variables is $\frac{\theta}{3}$, so the correlation between $A$ and $B$ ranges from $-\frac{1}{3}$ to $\frac{1}{3}$.

\subsection{Optimal scheme under FGM Copulas}

In this section, we assume that the cost distributions are restricted to FGM copulas, as shown in Eq.~(\ref{eq: FGM Copula}),
that is, 
$\calG = \{ G_\theta \colon -1 \leq \theta \leq 1\}$.
For each $w$, we numerically solve for the optimal combination of rewarding rule $(x, y, z)$ and the cost correlation $\theta$. These numerical computations are implemented in Mathematica (code available from the authors upon request). The results are summarized in Figure \ref{fig:2}.

\begin{figure}[ht]
\makebox[\textwidth][c]{
\begin{minipage}{.6\textwidth}
\centering
\includegraphics[width=\linewidth]{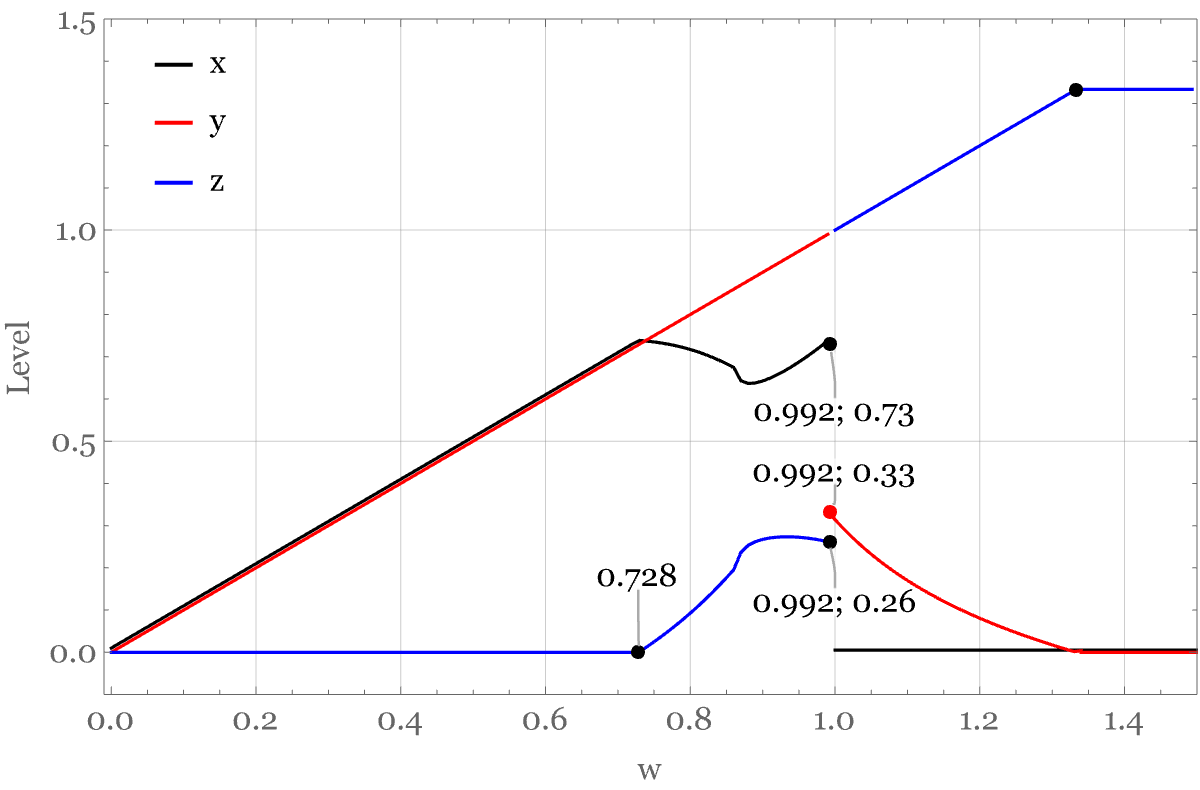}
(a) Optimal $x,y$ and $z$
\end{minipage}
\hspace{2mm}
\begin{minipage}{.6\textwidth}
\centering
\includegraphics[width=\linewidth]{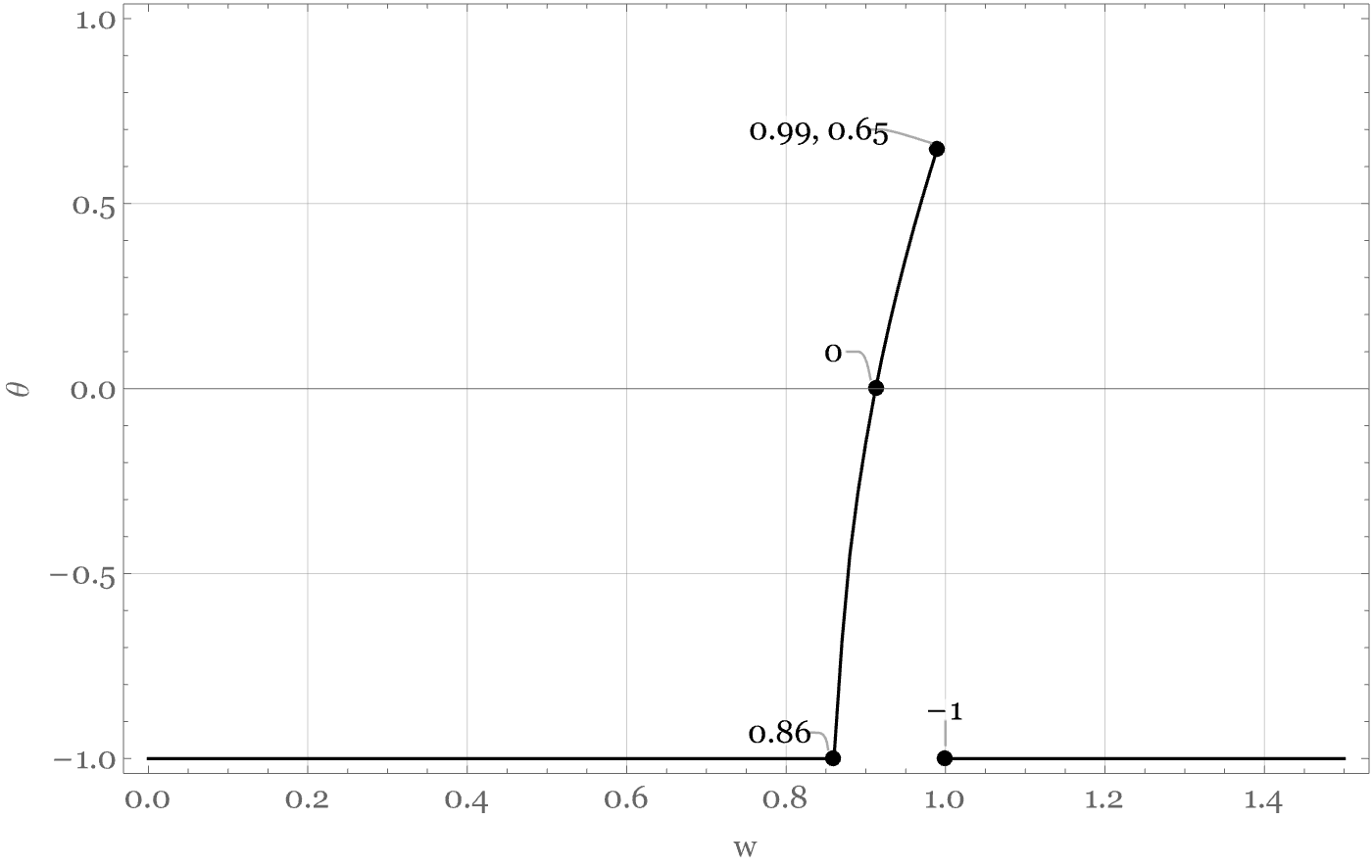}
(b) Optimal $\theta^*$
\end{minipage}}
\caption{The optimal scheme under FGM copulas family.}
\label{fig:2}
\end{figure}

Figure \ref{fig:2}(a) and Figure \ref{optimalRewards} reveal a similar pattern. When \( w < 1\), the principal emphasizes sufficient performance targeting by choosing large values for \( x \) and \( y \) and a small value for \( z \). Conversely, when \( w > 1\), the focus shifts to sustained performance targeting, with \( x \) and \( y \) relatively small and \( z \) large.

In line with the i.i.d.~case, for sufficiently small \( w < 1 \), the principal provides no reward for the second performance (\( z = 0 \)), corresponding to a purely sufficient rewarding rule. As \( w \) increases but remains under 1, the principal assigns a positive \( z \) to avoid completely neglecting the second performance, resulting in a partially sufficient rewarding rule, with \( z \) rising as \( w \) grows. A parallel logic applies when \( w > 1 \). If \( w \) is sufficiently large, the first performance receives no reward (\( x = y = 0 \)), so the optimal rewarding rule is purely sustained. For \( w \) just above 1, the principal finds it beneficial to set \( y > 0 \) for agents who skip the first period but perform in the second, producing a partially sustained rewarding rule.

Beyond choosing the rewarding rule, the principal can also vary the cost dependence parameter \(\theta\). Our analysis shows that for both relatively low budgets (\( w < 0.86\)) and relatively high budgets (\( w > 1\)), setting \(\theta = -1\) is optimal. In the intermediate range \((0.86, 1)\), a moderate level of \(\theta\) maximizes performance. To build intuition for this result, we first fix the rewarding rule to be either purely sufficient or purely sustained, and examine how \(\theta\) affects the agent’s likelihood of performing.

\begin{proposition}\label{fgminedoures}
Suppose \(G\) is defined by FGM copulas with parameter \(\theta\).  

\noindent(\RNum{1}) Under the purely sufficient rewarding rule \((x = w,\, y = w,\, z = 0)\), the overall performance decreases in \(\theta\).  

\noindent(\RNum{2}) Under the purely sustained rewarding rule \((x = 0,\, y = 0,\, z = w)\):  

   (a) the overall performance increases in \(\theta\) for \(w < 1\);
   
   (b) the overall performance decreases in \(\theta\) for \(w > 1\).  
\end{proposition}

\begin{proof}
    See Appendix \ref{pfgminedoures}
\end{proof}

Under the purely sufficient rewarding rule \(\,(x = w,\, y = w,\, z = 0)\), only the agent’s first performance is rewarded. In this setting, a negative correlation is advantageous for two reasons. First, it removes the low-cost agent’s expectation of having a second chance to perform, thereby eliminating any incentive to delay. Second, a negative correlation increases the likelihood that the agent’s cost in at least one period falls below \( w \), which in turn raises the probability of securing at least one performance.

Under the  purely sustained rewarding rule \(\,(x = 0,\, y = 0,\, z = w)\), only the second performance is rewarded. When \( w < 1\), the total reward is small, so only an agent with a sufficiently low cost in period 1 will perform. In this case, a \emph{positive} correlation helps by pairing low costs with other low costs, thereby increasing the propensity to perform across both periods. Conversely, when \( w > 1\), agents with higher costs in period 1 may also participate. As correlation becomes more positive, these higher-cost agents are more likely to face high costs again, making them less inclined to perform twice. A negative correlation, however, pairs these higher costs with lower costs, reducing the total cost of performing twice and thus encouraging more participation.

Figure~\ref{fig:2} illustrates that for very small \(w\), the principal’s optimal scheme mirrors the purely sufficient rule, while for large \(w\), it resembles the purely sustained rule. By Proposition~\ref{fgminedoures}, negative correlation (\(\theta = -1\)) is therefore preferable at both extremes. In the intermediate range \(0.86 < w < 1\), however, the principal blends the two reward structures. In this regime, negative correlation benefits the “sufficient” component, whereas positive correlation benefits the “sustained” component. The competing effects of these two forces yield an intermediate \(\theta\) as optimal, with \(\theta = 0\) being nearly optimal (see Figure~\ref{fig: optimal association parameter}).

\begin{figure}[ht]
\includegraphics[scale=0.6]{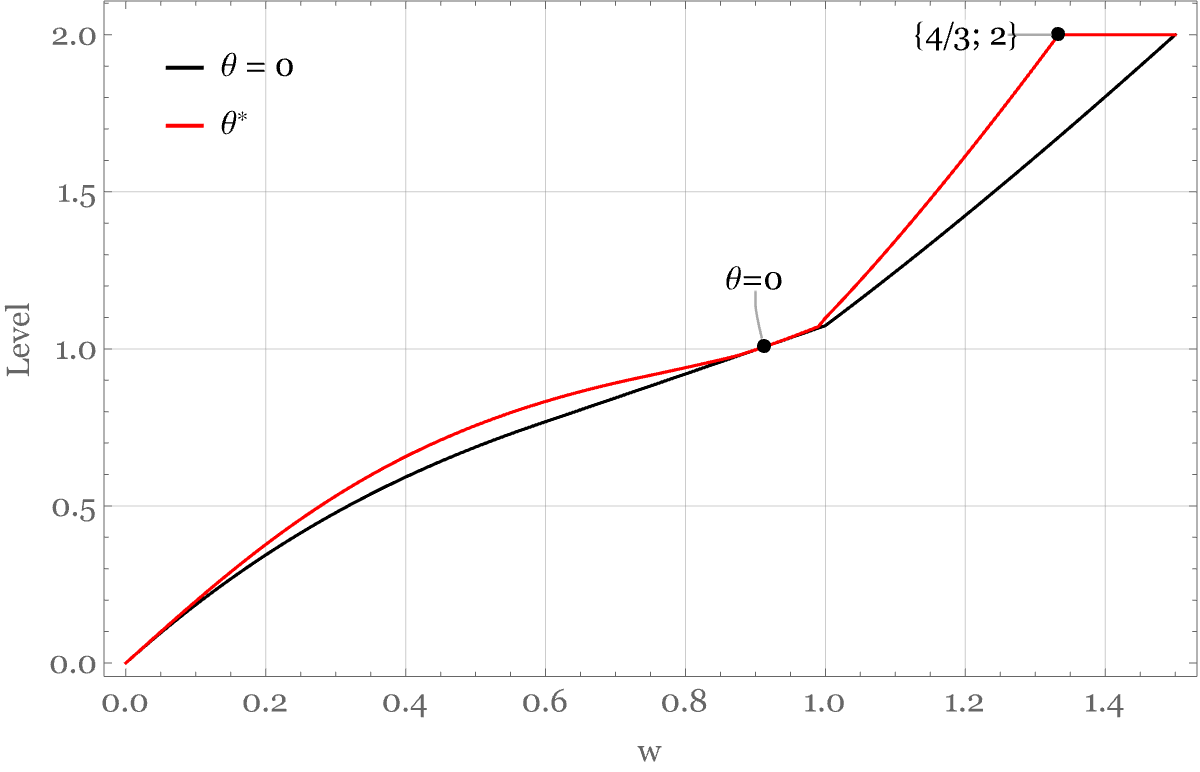}
\centering
\caption{Performance level under the optimal $\theta^*$ and under $\theta=0$.}
\label{fig: optimal association parameter}
\end{figure}


\section{Relaxing Monotonic Cost Correlations.}\label{optsection}

In the previous section, we examined a specific family of cost correlations that vary monotonically between the two periods. We now relax this assumption and allow the principal to select any joint cost distribution \(G\), subject only to the requirement that its marginal distributions on each period’s cost are uniform over \([0,1]\). This unrestricted case is mainly of theoretical interest, as it gives the principal an amount of flexibility over the dependence structure that is unlikely to be available in most practical applications.
Formally,
$\calG$ is the set of all cost distributions whose two marginals are $U[0,1]$.

Recall that a \emph{scheme} is defined as a pairing of the rewarding rule \((x, y, z)\) with a joint cost distribution \(G\). We show that under the optimal scheme, the rewarding rule must be either purely sufficient \(\bigl(x = w,\, y = w,\, z = 0\bigr)\) or purely sustained \(\bigl(x = 0,\, y = 0,\, z = w\bigr)\). The corresponding cost distribution \(G\) can exhibit a rather intricate structure.

We now introduce two families of schemes.
The first, termed a \emph{purely sufficient scheme}, seeks to induce as many agents as possible to perform exactly once, while the second, called a \emph{purely sustained scheme}, aims to have as many agents as possible perform twice.
We will show that any optimal scheme must belong to one of these two families. Let \(B_c\) denote the random variable \(B\) conditioned on \(A=c\).\footnote{We here abuse notations for clarity. Formally, $B_c$ is a random variable whose distribution coincides with (a version of) the conditional distribution of $B$ given $A=c$.}

\begin{definition}\label{onedoubledef1}
A scheme $\big(G,(x,y,z)\big)$ is called a \emph{purely sufficient scheme}
if $(x=w, y=w, z=0)$ and the joint distribution $G$ satisfies $P(B_c \geq  w) = 1$ for 
almost
every $c \in [0, w]$.
\end{definition}

Under a purely sufficient scheme, the principal allocates the full reward to the first performance. The corresponding cost distribution requires that the agent with types below $w$ in period 1 has a type above $w$ in period 2. This assumption further implies that the agent with a type below $w$ in period 2 had a type above $w$ in the previous period. In each period, an agent with a cost at most $w$ performs, and the overall performance level attains the upper-bound $2w$. 

Note that a purely sufficient scheme is feasible only when $w\leq \frac{1}{2}$. The following example illustrates such a construction for \(w = 0.4\).

\begin{example}\label{w04negex}
Let \(w = 0.4\). We partition the interval \([0,1]\) into two subintervals: \(I_1 = [0,0.4]\) and \(I_2 = (0.4,1]\). If an agent’s period 1 cost \(c\) lies in \(I_1\), then in period 2 his cost is drawn uniformly from \(I_2\), i.e., \ \(B_c \sim U(I_2)\). If \(c\) lies in \(I_2\), then with probability \(\tfrac{2}{3}\) his period 2 cost is drawn uniformly from \(I_1\), and with probability \(\tfrac{1}{3}\) it is drawn from \(I_2\). Figure~\ref{dengexample1} 
displays
the resulting density function of \(G\). 

\begin{figure}[ht]
\includegraphics[scale=0.8]{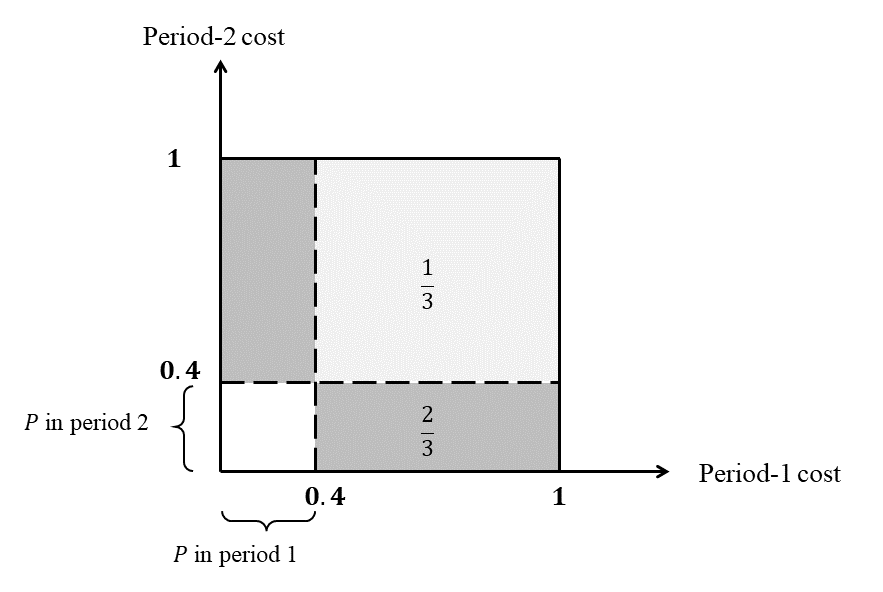}
\centering
\caption{The density function of $G$ in Example \ref{w04negex}.}\label{dengexample1}
\end{figure}
\end{example}


We now define a second family of schemes that focus on inducing two performances.

\begin{definition}\label{onedoubledef2}
A scheme $\big(G,(x,y,z)\big)$ is called a \emph{purely sustained scheme}
if $(x=0, y=0, z=w)$ and the joint distribution $G$ satisfies $B_c=w-c$ for every $c\in [0, w]$.
\end{definition}

Under purely sustained schemes, no reward is given for a single performance; instead, the principal awards the full prize only if the agent performs twice. The specified cost distribution ensures that an agent with period 1 cost \(c < w\) must have period 2 cost \(B_c = w - c\). Consequently, such an agent knows that the total cost of performing in both periods is exactly \(w\). Hence, under the rewarding rule \((x = 0,\, y = 0,\, z = w)\), all types below \(w\) perform twice, and all types above \(w\) do not perform at all, yielding a total performance level of \(2w\).

Unlike a purely sufficient scheme, a purely sustained scheme is feasible for any \(w > 0\). The next example illustrates one such construction for \(w = 0.9\). Note that for an agent whose period 1 cost exceeds \(w\), the period 2 cost may be assigned arbitrarily (provided that the marginal distribution remains uniform), as the agent will not perform in either period.

\begin{example}\label{w09negex}
Let \(w = 0.9\). Partition the interval \([0,1]\) into \(I_1 = [0,\,0.9]\) and \(I_2 = (0.9,\,1]\). If \(c \in I_1\), then \(B_c = 0.9 - c\). If \(c \in I_2\), then \(B_c \sim U(I_2)\). Figure~\ref{dengexample3}  
displays
the density function of \(G\).

\begin{figure}[ht]
\includegraphics[scale=0.8]{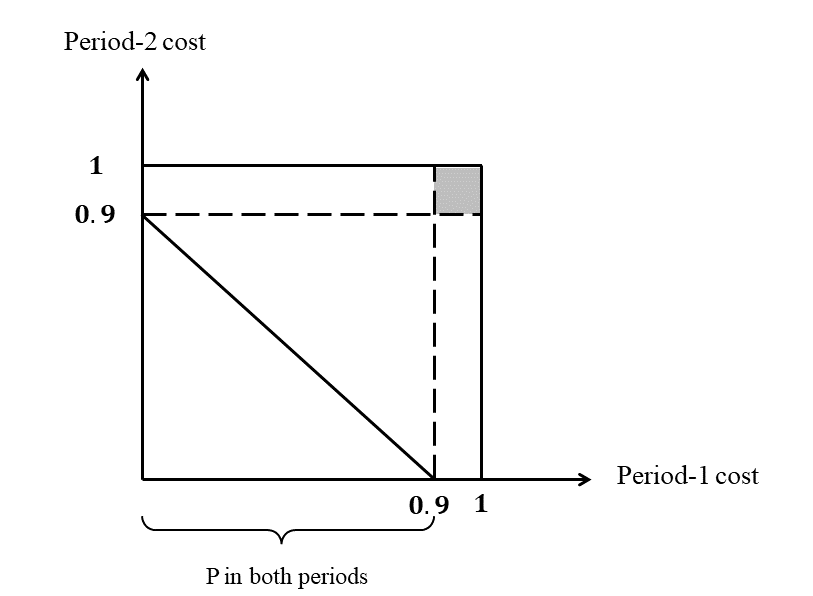}
\centering
\caption{The density function of $G$ in Example \ref{w09negex}.}\label{dengexample3}
\end{figure}
\end{example}

We are now ready to characterize the optimal scheme. For exposition reasons we focus on $w\leq 1$. As we will show later, for $w>1$, full performance in both periods can be attained, and the optimal scheme can take various forms.

\begin{proposition}\label{opuni}
For \( w \in \bigl[0, \tfrac{1}{2}\bigr] \), a scheme is optimal if and only if it is either a purely sufficient scheme or a purely sustained scheme.  
For \( w \in \bigl(\tfrac{1}{2}, 1\bigr] \), a scheme is optimal if and only if it is a purely sustained scheme.
\end{proposition}

\begin{proof}
See Appendix \ref{popuni}.
\end{proof}

This result aligns with Proposition \ref{fgminedoures}. Under the rewarding rule \(\,(x = w,\, y = w,\, z = 0)\) of a purely sufficient scheme, the performance level decreases monotonically with the cost dependence level. Hence, a negative correlation, where types below \( w \) swap with types above \( w \) across periods, maximizes performance. Conversely, under the rewarding rule \(\,(x = 0,\, y = 0,\, z = w)\) of a purely sustained scheme, both positive and negative correlations can help. Positive correlation is advantageous because a low type in period 1 is more likely to be low in period 2, making the agent more inclined to perform in both periods. Negative correlation is also beneficial because it balances the agent’s total cost of performing twice, thereby using the reward more efficiently.

The construction of a purely sustained scheme in Definition \ref{onedoubledef2} optimally combines these two effects. On the one hand, if an agent has a type below \(w\) in period 1, then he also has a type below \(w\) in period 2 (mimicking positive correlation). Furthermore, conditional on having a type below \(w\), a relatively lower (resp.\ higher) cost \(c\) in period 1 is paired with a relatively higher (resp.\ lower) cost \(w-c\) in period 2 (mimicking negative correlation). This finely tuned dependence makes the purely sustained scheme optimal for every budget level \(w \in [0,1]\). At the same time, this construction relies on a highly specific joint structure of costs that is unlikely to be observed in practice. We therefore view the purely sustained scheme mainly as a theoretical benchmark rather than a practical guideline.

Because both purely sufficient and purely sustained schemes can achieve the upper bound of \(2w\) for \( w \le \tfrac{1}{2} \), they are both optimal. We now sketch why no other scheme can achieve this upper bound.

We first argue that 
under an optimal scheme, either $z=w$ or $z=0$.
Indeed, if
$z\in(0, w)$, 
it is impossible to induce all types below $w$ to perform in both periods. Specifically:
    \begin{itemize}
    \item If $y<w$, then types $c\in\big(\max(y, z), w\big)$  do not perform in period 2; 
\item If $y=w$, then types $c\in(x, w)$  do not perform in period 1. \black 
\end{itemize}

\underline{Case 1:} $z=w$. We next show that to ensure all types below \( w \) perform in both periods, the scheme must be a purely sustained scheme (Definition \ref{onedoubledef2}). We first argue that 
in an optimal scheme $y$ must be 0.
Indeed, if $y > 0$,
then an agent with type $c$ close to $w$ (from below) does not perform in period 1:
\begin{itemize}
    \item If $B_c$ is high, then the agent does not perform in both periods.
    \item If $B_c$ is low, then the agent delays his performance to period 2.
\end{itemize}

Therefore, the optimal rewarding rule must be \( (x = 0, y = 0, z = w) \). Under this rule, to ensure the agent performs in both periods, the cost correlation must take the form \( B_c = w - c \) for \( c \in [0, w] \). 

To illustrate the intuition behind this, we consider the case where agent types are drawn from a discrete set (rather than the continuous interval \( [0,1] \)). While the case with continuous types is more technically complex, the underlying logic remains the same.

Take, for instance, the simple case where the type of the agent can take the values $\{0, 0.1, 0.2, \dots, 0.9, 1\}$. \black
Suppose, e.g., that $w=1$,
so that \black the rewarding rule is $(x=0, y=0, z=1)$, that is, an agent gets a reward 1 if and only if he performs twice. Which cost correlation can guarantee that all types perform in both periods? 
Clearly, an agent with type $c=1$ in period 1 performs if and only if he knows that his period 2 type is 0 with probability 1. 
Therefore, an agent with type $c=0.9$ in period 1 performs if and only if he knows that with certainty his period 2 type is $0.1$, etc. 

\underline{Case 2:} $z=0$. In this case, no reward goes to the second performance, so total performance is bounded above by \(\min(1,\, 2w)\).
Therefore, $z=0$ can be optimal only when $w\leq \frac{1}{2}$. 
In this case, at each stage all types with cost lower than $w$ must perform,
and each type performs at most once (since $z=0$).
Hence we must have $x=y=w$, and
the optimal cost correlation adopts the form shown in Definition \ref{onedoubledef1}.

\begin{remark}
For \( w = 1 \), a purely sustained scheme achieves full performance, making it optimal. This scheme remains optimal for all \( w > 1 \). However, when \( w > 1 \), it is possible to set \( z \) to be greater than 1, and additional cost structures may also become part of the optimal rewarding rules.
\end{remark}

\section{Conclusion}
\label{sec: conclusion}

In this paper, we examined the optimal combination of cost dependencies and reward schemes in a two-period principal-agent model subject to a budget constraint. When the principal selects a cost structure from a predefined, monotonically ranked family of cost dependencies, the total budget plays a central role. A low budget makes it optimal to focus on \emph{sufficient performance targeting}, that is, to concentrate rewards on inducing a single performance. By contrast, a high budget calls for \emph{sustained performance targeting}, that is, a back-loaded scheme that aims to induce two performances over time.

We also showed that negative correlation between the agent’s costs can be valuable for both low and high budgets. When the budget is low, a negative correlation removes the low-cost agent’s expectation of a second, potentially cheaper opportunity, thus preventing any incentive to delay performance. It also increases the chance that at least one period’s cost will be low enough to justify performing. When the budget is high, a negative correlation helps balance the agent’s total cost of performing twice, raising the likelihood of performing in both periods and improving the efficiency of reward allocation.

These insights are relevant for real-world environments in which principals operate under budget constraints, such as research funding agencies. In such settings, a principal who selects recipients whose cost structures are well aligned with the available resources and tailors the reward scheme accordingly can better motivate agents and enhance overall performance. 

Beyond monotonically ranked cost dependencies, we considered a more general setting in which the principal can select any joint cost distribution with uniform marginals. We provided a complete characterization of all optimal schemes. The qualitative result is consistent with the independent and copula cases: a low budget admits an optimal scheme that concentrates rewards on inducing a single performance, whereas a high budget calls for a purely sustained scheme that pays only after two performances. For low budgets, optimal schemes are not unique, but the main implication remains.

Our analysis has focused on a two-period setting. A natural question is whether the same insights extend to longer horizons. In a three-period version of the model in which the principal can choose both the reward scheme and the joint cost distribution, we find a similar pattern: under a very low budget, there is an optimal scheme that concentrates all rewards on inducing a single performance; as the budget increases, there is an optimal scheme that targets performance in two periods; and once the budget is sufficiently large, there is an optimal scheme that induces performance in all three periods. Characterizing the optimal reward scheme for intermediate budgets and for horizons longer than three periods remains an open question for future research.

\bibliographystyle{jpub2}
\bibliography{literature}

@unpublished{GuoHorner2020,
author  = {Yingni Guo and Johannes Hörner},
title   = {Dynamic Allocation without Money},
year    = {2020},
note    = {TSE Working Paper, n. 20-1133},
}

@article{BraunsteinBercovitz2013,
  author  = {Hedva Braunstein-Bercovitz},
  title   = {A multidimensional mediating model of perceived resource gain, work--family conflict sources, and burnout},
  journal = {International Journal of Stress Management},
  year    = {2013},
  volume  = {20},
  number  = {2},
  pages   = {95--115},
  doi     = {10.1037/a0032948}
}

@article{LivingstonJudge2008,
  author  = {Beth A. Livingston and Timothy A. Judge},
  title   = {Emotional responses to work--family conflict: An examination of gender role orientation among working men and women},
  journal = {Journal of Applied Psychology},
  year    = {2008},
  volume  = {93},
  number  = {1},
  pages   = {207--216},
  doi     = {10.1037/0021-9010.93.1.207}
}

@article{leshem2022option,
  title={The Option Value of Record-Based Sanctions},
  author={Leshem, Shmuel and Tabbach, Avraham},
  journal={Games and Economic Behavior},
  year={2023},
  volume={137},
  pages={1--22},
  publisher={Elsevier}
}

@incollection{Conway2014FGM,
  author       = {Conway, D. A.},
  title        = {Farlie--Gumbel--Morgenstern Distributions},
  booktitle    = {Wiley StatsRef: Statistics Reference Online},
  editor       = {Balakrishnan, N. and Colton, T. and Everitt, B. and Piegorsch, W. and Ruggeri, F. and Teugels, J. L.},
  year         = {2014}
}

@article{pavan2014dynamic,
  title={Dynamic mechanism design: A myersonian approach},
  author={Pavan, Alessandro and Segal, Ilya and Toikka, Juuso},
  journal={Econometrica},
  volume={82},
  number={2},
  pages={601--653},
  year={2014},
  publisher={Wiley Online Library}
}

@article{battaglini2005long,
  title={Long-term contracting with Markovian consumers},
  author={Battaglini, Marco},
  journal={American Economic Review},
  volume={95},
  number={3},
  pages={637--658},
  year={2005},
  publisher={American Economic Association}
}

@article{battaglini2019optimal,
  title={Optimal dynamic contracting: The first-order approach and beyond},
  author={Battaglini, Marco and Lamba, Rohit},
  journal={Theoretical Economics},
  volume={14},
  number={4},
  pages={1435--1482},
  year={2019},
  publisher={Wiley Online Library}
}

@article{sannikov2008continuous,
  title={A continuous-time version of the principal-agent problem},
  author={Sannikov, Yuliy},
  journal={The Review of Economic Studies},
  volume={75},
  number={3},
  pages={957--984},
  year={2008},
  publisher={Wiley-Blackwell}
}

@article{holmstrom1987aggregation,
  title={Aggregation and linearity in the provision of intertemporal incentives},
  author={Holmstrom, Bengt and Milgrom, Paul},
  journal={Econometrica},
  pages={303--328},
  year={1987},
  publisher={JSTOR}
}

@article{lazear1981agency,
  title={Agency, earnings profiles, productivity, and hours restrictions},
  author={Lazear, Edward P},
  journal={The American Economic Review},
  volume={71},
  number={4},
  pages={606--620},
  year={1981},
  publisher={JSTOR}
}

@article{harris1982theory,
  title={A theory of wage dynamics},
  author={Harris, Milton and Holmstrom, Bengt},
  journal={The Review of Economic Studies},
  volume={49},
  number={3},
  pages={315--333},
  year={1982},
  publisher={Wiley-Blackwell}
}

@Book{hutchinson1990continuous,
  author  = {Hutchinson, Timothy P and Lai, Chin Diew},
  title   = {Continuous bivariate distributions, emphasising applications},
  year    = {1990},
  publisher = {Rumsby Scientific Publishing, Adelaide}
}

@article{Schmalensee1984,
  author  = {Schmalensee, Richard},
  title   = {Gaussian Demand and Commodity Bundling},
  journal = {Journal of Business},
  volume  = {57},
  number  = {1, Part 2},
  pages   = {S211--S230},
  year    = {1984},
  month   = {January},
}

@article{Long1984,
  author  = {Long, John B., Jr.},
  title   = {Comments on `Gaussian Demand and Commodity Bundling'},
  journal = {Journal of Business},
  volume  = {57},
  number  = {1, Part 2},
  pages   = {S231--S233},
  year    = {1984},
  month   = {January},
}

@article{McAfee1989,
  author  = {McAfee, R. Preston and McMillan, John and Whinston, Michael D.},
  title   = {Multiproduct Monopoly, Commodity Bundling, and Correlation of Values},
  journal = {The Quarterly Journal of Economics},
  volume  = {104},
  number  = {2},
  pages   = {371--383},
  year    = {1989},
  month   = {May},
}

@article{Chu2011,
  author  = {Chu, Chenghuan Sean and Leslie, Phillip and Sorensen, Alan T.},
  title   = {Bundle-Size Pricing as an Approximation to Mixed Bundling},
  journal = {American Economic Review},
  volume  = {101},
  number  = {1},
  pages   = {263--303},
  year    = {2011},
  month   = {February},
}

@article{ChenRiordan2013,
  author  = {Chen, Yongmin and Riordan, Michael H.},
  title   = {Profitability of Product Bundling},
  journal = {International Economic Review},
  volume  = {54},
  number  = {1},
  pages   = {35--57},
  year    = {2013},
  month   = {February},
}

@article{AuChen2021,
  author = {Au, Pak Hung and Chen, Bin R.},
  title = {Matching with Peer Monitoring},
  journal = {Journal of Economic Theory},
  volume = {192},
  pages = {105172},
  year = {2021},
  doi = {10.1016/j.jet.2020.105172}
}

@article{GreenTaylor2016,
  author  = {Green, Brett and Taylor, Curtis R.},
  title   = {Breakthroughs, Deadlines, and Self-Reported Progress: Contracting for Multistage Projects},
  journal = {American Economic Review},
  year    = {2016},
  volume  = {106},
  number  = {12},
  pages   = {3660--3699},
  doi     = {10.1257/aer.20151181}
}

@article{JofreBonetPesendorfer2014,
  author  = {Jofre-Bonet, Mireia and Pesendorfer, Martin},
  title   = {Optimal Sequential Auctions},
  journal = {International Journal of Industrial Organization},
  year    = {2014},
  volume  = {33},
  pages   = {61--71},
  doi     = {10.1016/j.ijindorg.2014.02.002}
}

@article{CisternasFigueroa2015,
  author  = {Cisternas, Gonzalo and Figueroa, Nicolas},
  title   = {Sequential Procurement Auctions and Their Effect on Investment Decisions},
  journal = {The RAND Journal of Economics},
  year    = {2015},
  volume  = {46},
  number  = {4},
  pages   = {824--843},
  doi     = {10.1111/1756-2171.12112}
}

@article{CheGale1998,
  author  = {Che, Yeon-Koo and Gale, Ian},
  title   = {Standard Auctions with Financially Constrained Bidders},
  journal = {The Review of Economic Studies},
  year    = {1998},
  volume  = {65},
  number  = {1},
  pages   = {1--21},
  doi     = {10.1111/1467-937X.00033}
}

\newpage
\appendix
\section{Online Appendix}

\subsection{Proof of Lemma \ref{nowastetwo}}\label{pnowastetwo}

Suppose, to the contrary, 
that
under an optimal rewarding rule $x+z<w$. Suppose the principal increases $z$ while keeping the value of $x$ and $y$ unchanged. 
The overall performance level given in (\ref{uniformula}) can be rewritten as 
\[
F(\overline{c})\cdot\big(1+F(z)\big)+\big(1-F(\overline{c}) \big)\cdot F(y)
= F(\overline c)(1+F(z)-F(y)) + F(y). 
\]
By 
Eq.~\eqref{eq: payoff for plaintiff},
$\overline c$  
increases in $z$,
and,
since $1-F(y) \geq 0$, the overall performance increases in $z$.

\subsection{Proof of Proposition \ref{uniformop}}\label{puniformop}

Some parts of the proof rely on the assumption that $x+z=w$,
but this is only w.l.o.g.
\color{black}

\noindent\underline{Case 1}: $w\leq 1$.

Since both \(y\) and \(z\) are bounded by \(w\) (and hence are no greater than 1), the period 1 cost threshold simplifies to
\begin{equation}
\label{equ:bar-c}
    \overline{c} = \overline{c}(x,y,z)=\frac{1}{2} \left(2 x-y^2+z^2\right).
\end{equation}

\underline{Step 1}: Under the optimal rewarding rule, either $y=w$ or $z=w$.

By Eq.~\eqref{equ:bar-c}, $\partial \overline{c}/\partial z=-\big(1-z \big)$, and $\partial \overline{c}/\partial y=-y$. Let ${\rm obj}(x,y,z)$ denote the overall performance level shown in (\ref{uniformula}). The partial derivative of the principal's objective function with respect to $y$ and $z$ are
\begin{equation}\label{partialobjy}
\frac{\partial \rm{obj}}{\partial y}(x,y,z)=1-\overline{c}-y\cdot \big(1+z-y\big),
\end{equation}
\begin{equation}\label{partialobjz}
\frac{\partial \rm{obj}}{\partial z}(x,y,z)=\overline{c}-\big(1-z\big)\cdot \big(1+z-y\big).
\end{equation}
It then follows that
\begin{equation}\label{partialsum}
\frac{\partial \rm{obj}}{\partial y}(x,y,z)+ \frac{\partial \rm{obj}}{\partial z}(x,y,z)=\big(z-y\big)^2.
\end{equation}
This implies that as long as $y\neq z$, simultaneously increasing $y$ and $z$ 
by the same amount
improves overall performance level. Therefore, the optimal rewarding rule must be a corner solution (with either $y=w$ or $z=w$), or an interior solution with $y=z$.

Suppose 
that 
under the optimal rewarding rule $y=z$. In this case $\overline{c}=x$ and $\textrm{obj}(x,y,z)=x+z=w$, regardless of the exact value of $(x, y, z)$. In particular, $(x=w, y=0, z=0)$ is one combination that attains the performance level of $w$. Suppose the principal increases $y$ while keeping $x=w$ and $z=0$. By Eq.~(\ref{partialobjy}), the overall performance level increases. A contradiction to the assumption that under the optimal rewarding rule $y=z$. Therefore, the optimal rewarding rule must be a corner solution with either $y=w$ or $z=w$.

\underline{Step 2}: The optimal rewarding rule.

By (\ref{uniformula}), (\ref{equ:bar-c}), and 
since
$x=w-z$, we can write the overall performance level as a function of $y$ and $z$. 
Supposing that
$y=w$ (resp., $z=w$), we can calculate the value of $z$ (resp., $y$) that maximizes the overall performance level
(\ref{uniformula}).
A simple calculation shows that the optimal solution is 
$\bigl(y=w, z=\max\big(0, g(w)\big)\bigr)$ 
(resp., $\bigl(y=\max\big(0, h(w)\big), z=w\bigr)$).
Comparing 
these
two values, we find that the former is optimal when $w\leq 1$.\\


\noindent
\underline{Case 2: \(w > 1\).}

When \(w\) exceeds 1, 
\(y\) or \(z\) (or both) may exceed 1, so the cost threshold \(\overline{c}\) need not follow the form in \eqref{equ:bar-c}.

\underline{Step 1}: 
In the optimal rewarding rule, $y \leq 1$.

Intuitively, if $y > 1$ then one can lower $y$ to 1.
This change does not affect an agent who did not perform in period 1, because that agent will perform in period 2, whether $y=1$ or $y>1$.
And this change makes not performing in period 1 less desirable, so it cannot cause an agent who performs in period 1 to not perform in period 2.
We here show that lowering $y$ in fact strictly improves overall performance.
\color{black}

If \(y>1\), then
\[
\overline{c} 
= x + \tfrac{1}{2} 
+ \int_{0}^{\min(z,1)} (z - c)\,\mathrm{d}c 
- y,
\]
which is strictly decreasing in \(y\). From \eqref{uniformula}, the overall performance level is
\[
1 + F(\overline{c}) \cdot F(z),
\]
and this is strictly increasing in \(\overline{c}\). Thus, decreasing \(y\) increases \(\overline{c}\) and in turn improves performance. Consequently, no optimal solution has \(y>1\).

\underline{Step 2}: The case \(z\geq1\) 

We argue that when $z \geq 1$, 
the rewarding rule $(x,y,z)$ yields the same performance as the rewarding rule $(\widehat x,y,\widehat z)$,
provided $\widehat z \in [1,w]$ and $\widehat x + \widehat z = w$.
Indeed, since $z,\widehat z \geq 1$, under both rewarding rules an agent who performed in period 1 will perform in period 2,
and hence, since $x+z=\widehat x+\widehat z$, under both the incentive to perform in period 1 is the same.

\underline{Step 3}: 
Solving the case 
 \(y \le 1\) and \(z \le 1\).
 
By an argument similar to that in Step~1 of Case~1, one can show that under this restriction, the optimal values of \(y\) and \(z\) 
are
$y=1$ or $z=1$.
A simple calculation then shows that 
the optimal rewarding rule is $\big(x=w-1, y=\max\big(0, h(w)\big), z=1\big)$. 
The result follows by Step~2.

\subsection{Proof of Proposition \ref{fgminedoures}}\label{pfgminedoures}

\noindent\emph{Proof of part (\RNum{1})}. Suppose the rewarding rule is $(x=w, y=w, z=0)$.
For an agent with type \(A=c\) who performs in period 1, the payoff is \(u_p:=w-c\), independent of \(\theta\). 
Let \(B_c\) denote the random variable \(B\) conditioned on \(A=c\). If the agent does not perform in period 1, 
his
expected payoff is $u_{np}:=\int_0^w (w-k)g_c(k)\, \rmd k$, where \(g_c\) denote the density function for the distribution of the agent’s period 2 cost, given that the agent’s period 1 cost \(c\), and it depends on $\theta$. 
Note that
\begin{align*}
u_{np}
&=
\int_0^1 (w-k)g_c(k)\, \rmd k - \int_w^1 (w-k)g_c(k)\, \rmd k\\
&= w - E[B_c] + \int_w^1 (k-w)g_c(k)\, \rmd k\\
&>w-E[B_c].
\end{align*}


Under FGM copulas, $E[B_c]=\frac{1}{2}+\frac{(2c-1)\cdot \theta}{6}$. When $c>\frac{1}{2}$ (resp.~$c<\frac{1}{2}$), $E[B_c]$ is smaller (resp.~larger) than $c$, and it gets closer to $c$ when $\theta$ increases. As a result, when $c>\frac{1}{2}$, $E[B_c]<c$ and we have $u_{np}>w-E[B_c]>w-c=n_p$, and hence the agent 
does not perform
in period 1.

We next consider $c<\frac{1}{2}$. In this case, the agent 
performs
in period 1 if and only if
\begin{equation}\label{compc12}
   \underbrace{w-c}_{u_p}>\underbrace{\int_0^{\min\{w, 1\}} (w-k)g_c(k)\, \rmd k}_{u_{np}}. 
\end{equation}
The LHS of \eqref{compc12} is independent of $\theta$. The RHS is 
\begin{equation}\label{fgmone1}
u_{np}= \left\{
  \begin{array}{l l}
   \frac{2w^2}{3}\cdot \left(\frac{3}{4}+(\frac{1}{2}-c)(\frac{3}{2}-w)\theta\right), & \quad \text{if $w<1$,}\\
   {\frac { \left(\frac{1}{2}-c \right) \theta}{3}}+w-{\frac{1}{2}}, & \quad \text{if $w\geq 1$}.
  \end{array} \right.
\end{equation}
Since $c<\frac{1}{2}$ and $w<\frac{3}{2}$, $u_{np}$ is increasing in $\theta$. As a result, condition (\ref{fgmone1}) is less likely to hold when $\theta$ increases. This implies that the performance threshold in period 1 is decreasing in $\theta$.

Denote by $c^*$ the agent's performance threshold in period 1. The overall performance level is
\begin{equation}
c^*+\int_{c^*}^1 P(B\leq w|A=c)\, \rmd c.
\end{equation}

The part $\int_{c^*}^1 P(B\leq w|A=c)\, \rmd c$ evaluates the population who have a type above $c^*$ in period 1, and a type below $w$ in period 2. Fixing $c^*$ when $w<1$, this value is decreasing in $\theta$; when $w\geq 1$, this value is independent of $\theta$. Recall that $c^*$ is decreasing in $\theta$. As a result, the overall performance level is decreasing in $\theta$.\\

\noindent\emph{Proof of part (\RNum{2})}. Suppose the rewarding rule is $(x=0, y=0, z=w)$.
In period 1, an agent with type $c$ performs if and only if 
\begin{equation}\label{doublefgm0}
-c+\int_0^{\min\{w, 1\}} (w-k)g_c(k)\, \rmd k\geq 0.
\end{equation}
When $w<1$, this condition is equivalent to
\begin{equation}\label{doublefgm1}
c\leq 
\frac{1}{2}\cdot \frac{3w^2+3\theta w^2-2\theta w^3}{3+3\theta w^2-2\theta w^3}
=: c^*_1.
\end{equation}
The RHS of (\ref{doublefgm1}) is the threshold type of performance in period 1. 
Simple algebraic manipulations show 
\color{black}
that for $w<1$, $c^*_1$ is lower than $\frac{1}{2}$, and it is increasing in $\theta$.

The overall performance level is
\begin{equation}
c^*_1+\int_0^{c^*_1}P(B\leq w|A=c)\, \rmd c.
\end{equation}
It can be verified that for a fixed $c^*_1$, the value $\int_0^{c^*_1}P(B\leq w|A=c)\, \rmd c$ is increasing in $\theta$ for $w<1$. 
In addition, 
\color{black}
$c^*_1$ is increasing in $\theta$. 
Hence
the overall performance level is increasing in $\theta$.\\

When $w\geq 1$, condition \eqref{doublefgm0} is equivalent to 
\begin{equation}\label{doublefgm2}
    c\leq \frac{1}{2}\cdot \frac{\theta+3w+3(w-1)}{\theta+3}=:c_2^*.
\end{equation}
It can be verified that for $w>1$, 
$c^*_2$ is larger than $\frac{1}{2}$ and decreasing in $\theta$.

Since $P(B\leq 1|A=c) = 1$ for every $c$,
when $w \geq 1$,
\color{black}
the overall performance level is
\begin{equation}
c^*_2+\int_0^{c^*_2}P(B\leq w|A=c)\, \rmd c = 2c^*_2,
\end{equation}
hence
decreasing in $\theta$.

\subsection{Proof of Proposition \ref{opuni}}\label{popuni}

Since the purely sustained scheme attains the upper bound of $2w$,
under any optimal scheme,
in each period the agents with cost smaller than $w$ must perform.%
\footnote{ In fact, it might happen that a measure 0 of those agents do not perform. 
To simplify the proof, we ignore this possibility,
and assume that under the optimal scheme,
all agents whose cost is smaller than $w$ perform.}
\color{black}

We first argue that $z\in(0, w)$ cannot be part of the optimal scheme.
\begin{itemize}
\item If $z\in(0, w)$, then it is impossible to have all types below $w$ perform in both periods:
\begin{itemize}
\item If $y<w$, then types $c\in\big(\max(y, z), w\big)$  do not perform in period 2; 
\item If $y=w$, then types $c\in(x, w)$  do not perform in period 1. This is because performing only in period 1 is worse than 
not performing
\color{black}
in both periods; and performing in both periods is strictly worse than performing only in period 2.
\end{itemize}
\end{itemize}
Therefore, the optimal rewarding rule must have either $z=w$ or $z=0$.\\

\underline{Case 1:} $z=w$. 

We 
show that to have all types below $w$ perform in both periods, the scheme must be a purely sustained scheme shown in Definition \ref{onedoubledef2}. We first argue that when $z=w$, 
we must have $y=0$ in an optimal scheme,
so that in this case the optimal rewarding rule must adopt the form $(x=0, y=0, z=w)$.

Consider an agent whose period 1 cost is $\widehat{c}\in (0, w)$. To attain the upper-bound performance level $2w$, this agent must perform in period 1. Since $\widehat{c}>x=0$, a single performance in period 1 is not beneficial for the agent. 
Therefore, the agent necessarily performs in period 2 with a positive probability. 

The agent will perform in period 2 only if his cost in period 2 is at most $w$.
Hence, the fact 
that the agent's expected cost of performing in both periods is at most
his expected gain, translates to:
\color{black}
\begin{equation}\label{nodevuniq0}
\widehat{c}+\int_0^w c\cdot g_{\widehat{c}}(c) \, \rmd c\leq \int_0^w w\cdot g_{\widehat{c}}(c) \, \rmd c=w\cdot G_{\widehat{c}}(w), 
\end{equation}
where $g_{\widehat{c}}(c)$ is the agent's period 2 cost conditional on $A=\widehat{c}$. As $\widehat{c}$ goes to $w$, the LHS of (\ref{nodevuniq0}) is at least $w$, while the RHS of (\ref{nodevuniq0}) is at most $w$. Therefore,
\begin{equation}\label{wglim}
\lim_{\widehat{c}\to w}\int_0^w w\cdot g_{\widehat{c}}(c) \, \rmd c=w,
\end{equation}
and
\begin{equation}\label{gcclim}
\lim_{\widehat{c}\to w}G_{\widehat{c}}(w)=1.
\end{equation}

The agent's gain from two periods' performances is  
\begin{equation}
w\cdot G_{\widehat{c}}(w)-\widehat{c}-\int_0^w c\cdot g_{\widehat{c}}(c) \, \rmd c.
\end{equation}  
If, instead, the agent does not perform in period 1 and only performs in period 2, his expected payoff is  
\begin{equation}
\int_0^y(y-c)\cdot g_{\widehat{c}}(c)\, \rmd c.
\end{equation}  

To have the agent performs in period 1, it must be the case that  
\begin{equation}\label{nodevuniq1}
\begin{split}
0 \geq 
\color{black}
&\int_0^y(y-c)\cdot g_{\widehat{c}}(c)\, \rmd c-\left(w\cdot G_{\widehat{c}}(w)-\widehat{c}-\int_0^w c\cdot g_{\widehat{c}}(c) \, \rmd c\right)\\
=&\, y\cdot G_{\widehat{c}}(y)+\int_y^w c\cdot g_{\widehat{c}}(c)\, \rmd c-\left(w\cdot G_{\widehat{c}}(w)-\widehat{c}\right).
\end{split}
\end{equation}  
By (\ref{gcclim}), $\lim_{\widehat{c}\to w}\big(w\cdot G_{\widehat{c}}(w)-\widehat{c}\big)=0$. 
For (\ref{nodevuniq1}) to hold, the sum \( y\cdot G_{\widehat{c}}(y)+\int_y^w c\cdot g_{\widehat{c}}(c)\, \rmd c \), which is non-negative, must also go to zero as $\widehat{c}$ increases to $w$. 

By (\ref{wglim}) and Integration by Parts,  
\begin{equation}\label{nodevuniq2}
\begin{split}
\lim_{\widehat{c}\to w}\int_0^w c\cdot g_{\widehat{c}}(c) \, \rmd c&=\lim_{\widehat{c}\to w}\left(w\cdot G_{\widehat{c}}(w)-0-\int_0^wG_{\widehat{c}}(c)\, \rmd c\right)\\
&=\lim_{\widehat{c}\to w}\int_0^w\big(G_{\widehat{c}}(w)-G_{\widehat{c}}(c)\big)\, \rmd c\\
&=0.
\end{split}
\end{equation}  

It follows that
$\lim_{\widehat{c}\to w} \int_y^w c\cdot g_{\widehat{c}}(c)\, \rmd c \leq \lim_{\widehat{c}\to w} \int_0^w c\cdot g_{\widehat{c}}(c)\, \rmd c = 0$.
We conclude that $y=0$ or 
$\lim_{\widehat{c}\to w}G_{\widehat{c}}(y)=0$.
In the latter case, by (\ref{gcclim}),
\color{black}
%
%
\begin{align}
\nonumber
\lim_{\widehat{c}\to w}\int_0^w\big(G_{\widehat{c}}(w)-G_{\widehat{c}}(c)\big)\, \rmd c
=&\lim_{\widehat{c}\to w}\left(\int_0^y\big(G_{\widehat{c}}(w)-0\big)\, \rmd c+\int_y^w\big(G_{\widehat{c}}(w)-G_{\widehat{c}}(c)\big)\, \rmd c\right)\\
\label{nodevuniq3}
\geq &\,\lim_{\widehat{c}\to w} y\cdot G_{\widehat{c}}(w)\\
=&y.
\nonumber
\end{align}

By (\ref{nodevuniq2}) and (\ref{nodevuniq3}),
$y=0$.

\bigskip

We have thus
\color{black}
shown that the optimal rewarding rule adopts the form $(x=0, y=0, z=w)$. We next analyze the optimal cost correlation. Under the optimal scheme with the rewarding rule \((x=0, y=0, z=w)\), an agent with type \(c\leq w\) performs in period 1 with probability 1. Therefore, with probability 1, \(E[B_c]\leq w-c\).
Taking expectations on both sides, we obtain that
\[\frac{w}{2}=E[E[B_c]]\leq E[w-c]=\frac{w}{2},\]
which implies that 
for almost all $c\in [0,w]$,
\color{black}
\begin{equation}\label{ycwceq}
    E[B_c]=w-c.
\end{equation}
We next show that in fact, \(B_c=w-c\) for almost all $c\in [0,w]$.

For a simple illustration we focus on the case \(w=1\) below. The case with \(w\in (0, 1)\) follows a similar argument and 
is discussed at the end of this section.

By the assumption that the 
period 1 and the period 2
costs are both uniformly distributed, we have \(A\sim U[0, 1]\) and \(B\sim U[0, 1]\). As a result, \(E[A]=E[B]=\frac{1}{2}\). We 
will
prove that the correlation coefficient \(\rho(A, B)=-1\). 
Indeed,
\begin{align}
\Cov(A, B)
&= E[AB] - E[A]\cdot E[B] \label{covxy-line1}\\
&= E\bigl[E[AB \mid A]\bigr] - E[A]\cdot E[A] \label{covxy-line2}\\
&= E\bigl[A \cdot E[B \mid A]\bigr] - (E[A])^2 \label{covxy-line3}\\
&= E\bigl[A \cdot (1 - A)\bigr] - (E[A])^2 \label{covxy-line4}\\
&= E[A] - E[A^2] - (E[A])^2 \label{covxy-line5}\\
&= \tfrac{1}{2} - \tfrac{1}{3} - \tfrac{1}{4} 
   = -\tfrac{1}{12} \label{covxy-line6}\\
&= -\Var(A)
   = -\Var(B), \label{covxy-line7}
\end{align}
where
\color{black}
Eq.~\eqref{covxy-line2} holds by the law of iterated expectation and the fact that \(E[A]=E[B]\); \eqref{covxy-line3} holds because \(A\) is constant given \(A\). Eq.~\eqref{covxy-line4} follow from \eqref{ycwceq}. Eq.~\eqref{covxy-line6} and \eqref{covxy-line7} hold since \(A\) has the uniform distribution.
Eqs.~\eqref{covxy-line1}–\eqref{covxy-line7} imply that
\[
\rho(A, B)=\frac{\Cov(A, B)}{\sqrt{\Var(A)}\cdot \sqrt{\Var(B)}}=-1,
\]
as claimed.\\

We now adapt the argument 
to the case \(w<1\). We will show that conditional on \(A\leq w\), the correlation coefficient \(\rho(A, B)\bigl|_{A\leq w}\) is \(-1\). 

The analog of \eqref{ycwceq} is \(E[B \mid A]\bigl|_{A\leq w}=w-A\). Since
\begin{equation}
\begin{split}
E[B \mid A \le w]
&=\int_{a=0}^w E[B \mid A=a]\, \rmd a
=\int_{a=0}^w (w - a)\, \rmd a
=\frac{w}{2},
\end{split}
\end{equation}
\[
E[A \mid A \le w]
=\int_0^w a \cdot \frac{1}{w}\, \rmd a
=\frac{w}{2},
\]
and
\[
E[A^2 \mid A \le w]
=\int_0^w a^2 \cdot \frac{1}{w}\, \rmd a
=\frac{w^2}{3},
\]
we have
\begin{equation}\label{covxy2}
\begin{split}
\Cov(A, B)\bigl|_{A\leq w}
&= E[AB \mid A \le w]
   - E[A \mid A \le w]\cdot E[B \mid A \le w]\\
&= E\!\bigl[E[AB \mid A] \;\big|\; A \le w\bigr]
   - E[A \mid A \le w]\cdot E[A \mid A \le w]\\
&= E\!\bigl[A \cdot E[B \mid A] \;\big|\; A \le w\bigr]
   - \bigl(E[A \mid A \le w]\bigr)^2\\
&= E\!\bigl[A \cdot (w - A) \;\big|\; A \le w\bigr]
   - \bigl(E[A \mid A \le w]\bigr)^2\\
&= w \cdot E[A \mid A \le w]
   \;-\; E[A^2 \mid A \le w]
   \;-\; \bigl(E[A \mid A \le w]\bigr)^2\\
&= \frac{w^2}{2} \;-\; \frac{w^2}{3} \;-\; \frac{w^2}{4}
   \;=\; -\frac{w^2}{12}\\
&= -\,\Var\bigl(A \mid A \le w\bigr)
   \;=\; -\,\Var\bigl(B \mid A \le w\bigr).
\end{split}
\end{equation}
Therefore,
\[
\rho(A, B)\bigl|_{A\leq w}
\;=\;
\frac{\Cov(A, B)\bigl|_{A\leq w}}
     {\sqrt{\,\Var(A \mid A \le w)}\;\sqrt{\,\Var(B \mid A \le w)}} 
\;=\;-1.
\]\\

\underline{Case 2:} $z=0$. In this case, the second performance is not rewarded,
and hence 
\color{black}
the total performance level is bounded above by 1. Therefore, $z=0$ can be optimal only when $w\leq \frac{1}{2}$.

Assume then that $w \leq \frac{1}{2}$.
Since under the optimal scheme, at each period all types with cost less than $w$ perform,
since only agents with cost less than $x$ (resp., $y$)
may perform in period 1 (resp., period 2),
and since $x,y\leq w$,
in the optimal scheme we have $x=y=w$.
Since an agent performs at most once,
the optimal cost correlation adopts the form shown in Definition \ref{onedoubledef1}.

\end{document}